\documentclass[10pt]{article}
\usepackage[latin1]{inputenc}
\usepackage{style-paper}

\usepackage{authblk}
\title{On the Ground State Energies of Discrete and Semiclassical Schr\"odinger Operators}
\author{Isabel Detherage}
\author{Nikhil Srivastava}
\author{Zachary Stier}
\affil{Department of Mathematics, UC Berkeley}
\date{\today}

\usepackage{fullpage}

\begin{document}

\maketitle

\begin{abstract}
    We study the infimum of the spectrum, or ground state energy (g.s.e.), of a discrete Schr\"odinger operator on $\theta\Z^d$ parameterized by a  potential $V:\R^d\rightarrow\R_{\ge 0}$ and a frequency parameter $\theta\in (0,1)$. We relate this g.s.e. to that of a corresponding continuous semiclassical Schr\"odinger operator on $\R^d$ with parameter $\theta$, arising from the same choice of potential. We show that:
    \begin{enumerate}[label=(\roman*)]
        \item The discrete g.s.e.\ is at most the continuous one for continuous periodic $V$ and irrational $\theta$. 
        \item The opposite inequality holds up to a factor of $1-o(1)$ as $\theta\rightarrow 0$ for sufficiently regular smooth periodic $V$. \item The opposite inequality holds up to a constant factor for every bounded $V$ and $\theta$ with the property that discrete and continuous averages of $V$ on fundamental domains of $\theta \Z^d$ are comparable. 
    \end{enumerate}
    Our proofs are elementary and rely on sampling and interpolation to map low-energy functions for the discrete operator on $\theta \Z^d$ to low-energy functions for the continuous operator on $\R^d$, and vice versa.
\end{abstract}

\section{Introduction}
Let $V:\R^d\rightarrow\R_{\ge 0}$ be a smooth $\Z^d$-periodic potential and let $\theta\in (0,1)$ be a parameter. We study the family of discrete Schr\"odinger operators $A_\theta: \ell_2(\gt\Z^d) \to \ell_2(\gt\Z^d)$ defined by
\begin{equation}
    \label{eqn:atheta}
    A_\theta f(k\gt) := \left(2df(k\gt) - \sum\limits_{j \sim k} f(j\gt)\right) + V(k\theta)f(k\gt)=(L  + V) f(k\theta),\qquad \text{for $k\in \Z^d$}
\end{equation}
where $j\sim k$ indicates that $j,k\in \Z^d$ differ in exactly one coordinate and $L$ is the discrete Laplacian of the corresponding graph on $\theta \Z^d$. This family includes several well-studied operators, such as the almost-Mathieu operator which corresponds to the special case 
\begin{align}
    \label{eqn:almostmathieu}  V(x):=\lambda(2-2\cos(2\pi  x)), \qquad \text{for $\lambda>0$}
\end{align}
in $d=1$. The  operator $A_\theta$ is bounded and self-adjoint with $\spec(A_\theta)\subseteq [0,4d+\norm{V}_\infty]$. We refer to the infimum of its spectrum as its {\em ground state energy}, denoted $\mu(A_\theta)$, which admits the well-known variational characterization
\begin{equation}
    \label{eqn:avar}
    \mu(A_\theta) = \inf_{\substack{f\in \ell_2(\theta \Z^d) \\ \norm{f}=1}} \langle f,A_\theta f\rangle = \inf_{\substack{f\in \ell_2(\theta \Z^d) \\ \norm{f}=1}} \lpr{\sum_{j\sim k\in\Z^d} (f(k\theta)-f(j\theta))^2 + \sum_{k\in \Z^d} V(k\theta)f(k\theta)^2}.
\end{equation}

Motivated by questions in geometric group theory, there has been  interest \cite{beguin1997spectrum,boca2005norm,bump2017exercise,bump2017useful,ozawa2022substitute} in proving good upper and lower bounds on $\mu(A_\theta), \theta\in\lbr{0,\half}$ for the special case \eqref{eqn:almostmathieu} --- roughly speaking, this quantity arises naturally in the representation theory of the Heisenberg group and controls the spectral gap of the random walk on that group. More recently, a line of work initiated by Ozawa \cite{ozawa13, kaluba2019aut, kaluba2021property} introduced a method for producing semidefinite programming proofs of Kazhdan's Property (T) for various groups including $SL_n(\Z)$, in which effective bounds on certain $\mu(A_\theta)$ play a role. The best known bounds on $\mu(A_\theta)$ for the almost-Mathieu operator due to \cite{boca2005norm} rely on delicate trigonometric calculations, and as far as we are aware there is no systematic method for proving such bounds in general. 

In contrast, in many cases the ground state energy of a corresponding periodic Schr\"odinger operator on $\R^d$ can be estimated using well-established tools of spectral theory --- see e.g.\ \cite{kato1952note,gesztesy1992ground} for examples of such bounds in $1$ dimension, and \propref{prop:muBlb intro} for coarse bounds in any dimension. In this context, consider the {$d$-dimensional Schr\"odinger operator with potential $V$ and semiclassical parameter $\gt$}, $B_\theta:H^2(\R^d)\to L^2(\R^d)$, defined as
\begin{equation}
    \label{eqn:btheta}
B_\theta f(x) := -\theta^2 \sum\limits_{i=1}^d \frac{d^2}{dx_i^2} f(x) + V(x)f(x)=(-\gt^2\Delta+V)f(x).
\end{equation}
$B_\gt $ is an unbounded selfadjoint nonnegative operator. Let $\mu(B_\gt)$ denote the ground state energy of $B_\theta$. We then have the variational characterization (see e.g.\ \cite[Chapter VIII]{reed1981functional})
\begin{equation}
    \label{eqn:bvar}
    \mu(B_\gt) =  \inf_{\substack{g\in H^2(\R^d) \\ \norm{g}=1}} \langle g,B_\theta g\rangle = \inf_{\substack{g\in H^2(\R^d) \\ \norm{g}=1}} \int_{\R^d} \gt^2\norm{\nabla g(x)}^2 + V(x)g(x)^2\d x.
\end{equation}

The purpose of this paper is to establish a relationship between $\mu(A_\theta)$ and $\mu(B_\theta)$ for a wide class of $V$, allowing bounds in one setting to be transferred to the other. To our knowledge, this is the first work to relate the two quantities. Our results are general, in that they apply for a wide class of potentials, and effective, in that they provide concrete bounds for fixed $\theta$ as is required in the applications to group theory mentioned above. 

We prove the following three results which articulate that while the spectra of $A_\theta$ and $B_\theta$ are clearly different, their ground state energies are in many cases comparable. With the exception of \thmref{thm:sampling}, we do not expect the bounds obtained here to be sharp in terms of their dependence on $\theta$. 

\paragraph{Upper bound on $\mu(A_\theta)$ for continuous periodic $V$, irrational $\theta$.}

\begin{theorem}\label{thm:sampling}
    Let $V:\R^d \to \R _{\ge 0}$ be continuous and $\mathbb{Z}^d$-periodic. For $\gt$ irrational,
    $$\mu(A_\theta) \le \mu(B_\theta).$$
\end{theorem} 
The irrationality assumption above cannot be removed in general --- see \rmkref{rem:irrat} for a discussion. However the above result can be extended to all $\gt$ by instead considering the infimum of the {\em union spectrum} of $A_\theta$ (see \defref{def:union}) instead of $\mu(A_\gt)$.  The proof of \thmref{thm:sampling} appears in \secref{sec:upper}.

\paragraph{Lower bound on $\mu(A_\theta)$ for smooth periodic $V$, small $\theta$.} Under stronger assumptions on $V$, we establish the reverse inequality $\mu(A_\theta)\ge \mu(B_\theta)(1-o(1))$ as $\theta\rightarrow 0$. To state our result we introduce the following definitions.

\begin{definition}[Critical Point] \label{def:critical}
    For a multi-index $\ga=(\ga_1,\dots,\ga_d)\in\N^d$, let $D^\ga:=\lpr{\frac{\dd}{\dd x_1}}^{\ga_1}\cdots\lpr{\frac{\dd}{\dd x_d}}^{\ga_d}$. A point $m\in\R^d$ is a {\em critical point of order $p$} of $V$ if $D^\alpha V(m)=0$ for all $\ga$ satisfying $\norm{\alpha}_1\le p-1$, and  $D^\alpha V(m)\neq 0$ for some $\alpha$ satisfying $\norm{\alpha}_1=p$. 
\end{definition}
\begin{definition}[Coercive] \label{def:coercive}
A potential $V$ is {\em $P$-coercive} with parameters $(t_0, K, P)$, $K\ge 1$, if for every $t\le t_0$ and $y\in\R^d$, $V(y) \le t+\inf V$ implies that there exists a critical point $m$ of order $p$ with $2\le p\le P$ such that $\norm{y-m} \le K t^\frac{1}{p}$. We say that $V$ is {\em uniformly $p$-coercive} if it is $p$-coercive and all of its critical points have order exactly $p$. \end{definition} 

Coercivity guarantees that if a function is close to its minimum value, then its input must be polynomially close to a critical point. Note that any real analytic periodic potential with finitely many local minima satisfies this property. The bound below depends on the parameters of coercivity as well as $V$'s smoothness.

\begin{theorem}\label{thm:lowerbound}
    Let $V:\R^d \rightarrow \R _{\ge 0}$ be smooth, $\mathbb{Z}^d$-periodic, and $P$-coercive with bounded directional derivative in every sequence of $P$ directions --- i.e., for each $x\in\R^d$ and $\ga\in\N^d$ with $\norm{\ga}_1\le P$, there is an upper bound on $D^\ga V(x)$ independent of $\ga$ and $x$. Then for $\gt$ sufficiently small, 
    $$\lpr{1+O\lpr{\gt^{\frac{1}{2(P+1)}}}}\mu(A_\gt)\ge\mu(B_\gt)-O\lpr{\gt^{2-\frac{3P+1}{2P(P+1)}}}.$$
\end{theorem}
The proof of \thmref{thm:lowerbound} appears in \secref{sec:lowerasymp}, where a more precise version (\thmref{thm:viataylor}) specifying the constants in the asymptotic notation is provided.
The estimate above is nontrivial at least for $V$ with isolated critical points: for such potentials, it can be shown that the additive  $O(\gt^{2-\frac{3P+1}{2P(P+1)}})$ term on the right hand side is of smaller order than $\mu(B_\theta)$ as $\theta\rightarrow 0$. In particular we prove the following estimates in \secref{sec:sobolev proof} using simple arguments which we suspect are known but could not find in the literature.
\begin{proposition}\label{prop:muBlb intro}
    Suppose $V:\R^d\rightarrow\R_{\ge0}$ is a smooth $\Z^d$-periodic potential which is uniformly $p$-coercive for some $p\ge 2$ and has only isolated critical points. Then as $\theta\rightarrow 0$ we have
    \begin{enumerate}
        \item[(a)] $\mu(B_\theta)=\Omega\lpr{\theta^{\frac{2p}{p+1}}}$ if $d=1,2$,
        \item[(b)] $\mu(B_\theta)=\Omega\lpr{\theta^{\frac{2p}{p+2}}}$ if $d\ge 3$, and
        \item[(c)] $\mu(B_\gt)=O\lpr{\gt^{\frac{2p}{p+2}}}$ for all $d\ge1$,
    \end{enumerate}
    where the implicit constants depend on $V$.
\end{proposition}

\paragraph{Lower bound on $\mu(A_\theta)$ for bounded $V$ and any $\theta$.} For $y \in \gt \Z^d$, let $C(y) := y + \{0, \gt\}^d$ and $\ol{C}(y) := y + [0,\gt]^d$. We then have the following comparison between $\mu(A_\theta)$ and $\mu(B_\theta)$ under the less restrictive assumptions that $V$ is bounded, nonnegative, and  its local averages on the cubes $C(y)$ and $\ol{C}(y)$ are comparable for all $y\in \theta \Z^d$.

\begin{theorem}\label{thm:simple}
    Let $V: \R ^d \rightarrow \R _{\ge 0}$, not necessarily periodic with $\norm{V}_\infty\le 1$, and let $\theta\in (0,1)$. Suppose that there exists $a =a(\theta)>0$ such that for all $y\in\gt\Z^d$, 
    \begin{equation}\label{eqn:adef}\frac{1}{\gt^d}\int_{\ol{C}(y)}V(x) \d x \le a \cdot \frac{1}{2^d} \sum\limits_{k \in C(y)}V(k).\end{equation}
    Then, 
    $$c_d\mu(A_\theta)\ge\mu(B_\theta)$$ 
    where $c_d = \max \lcr{\frac{3^d}{2^d} + 5 \cdot 3^d, 4\cdot\frac{3^d}{2^d}\cdot a}$. 
\end{theorem}

Combined with \thmref{thm:sampling} this establishes that for $V$ satisfying the conditions of both theorems and $\gt$ irrational, $\mu(A_\gt)$ and $\mu(B_\gt)$ agree up to a multiplicative factor. The quantity $c_d$ arises as the worst-case relationship between local Rayleigh quotients on individual cubes $C(y)$ and $\ol{C}(y)$. The requirement that $a>0$ in particular rules out the possibility that $V(y)=0$ for all $y\in \theta \Z^d$ while having $V\neq 0$ elsewhere, in which case we could have $0=\mu(A_\theta)<\mu(B_\theta)$.
The proof of \thmref{thm:simple} appears in \secref{sec:lowerany}. 

\paragraph{Example.} As an example, consider the almost-Mathieu operator. The potential \eqref{eqn:almostmathieu} is coercive with parameters $\lpr{4\gl,\frac{1}{2\pi\sqrt{\gl}},2}$, its Lipschitz constant is $4\pi\gl$, and its second derivative is bounded by $8\pi^2\gl$. Therefore we find from \thmref{thm:viataylor} that
$$\mu(B_\gt)\le\mu(A_\gt)+ \frac{1}{1-\mu(A_\gt)} \lpr{8\pi\sqrt{\gl}\gt^{1+q/2} + \lpr{\mu(A_\gt)+\frac{8\pi\gl\gt^{1-q}}{1-4\pi\gl\gt^{1-q}}}\mu(A_\gt)}$$
for all $0<q<1$. This choice of $V$ also obeys the conditions of \propref{prop:muBlb intro}, so from the choice $q=\frac{5}{6}$ and the constant in \eqref{eq:muBub exact} being $\sqrt[4]{8}\pi\sqrt{\gl}$, we obtain the multiplicative upper bound
\begin{align*}
    \frac{\mu(B_\gt)}{\mu(A_\gt)}&\le1+\frac{1}{1-\sqrt[4]{8}\pi\sqrt{\gl}\gt^2}\lpr{8\pi\sqrt{\gl}\gt^{\frac{1}{12}}+\sqrt[4]{8}\pi\sqrt{\gl}\gt^2+\frac{8\pi\gl\gt^{\frac{1}{6}}}{1-4\pi\gl\gt^{\frac{1}{6}}}}.
\end{align*}
The lower bound of 1 follows from \thmref{thm:sampling}. We note that the large error bound of $O\lpr{\gt^{\frac{1}{12}}}$ (in comparison to \cite{witt}'s \eqref{eqn:wz}; see the following section for further discussion) is a lossy consequence of applying a Lipschitz bound for the potential on its entire domain. 

In the setting of Theorems \ref{thm:simple}/\ref{thm:exp bd}, when $\gt<\half$, $V$ has $a\le\frac{2}{3}$; note that if $\gt = \half$ such an $a$ does not exist. Thus $c_1$ as in the Theorem statement is $\frac{33}{2}$ (we note that this can be improved to $c_1 = 9$ with modified choices of parameters in the proof which are preferable in $d=1$).

\subsection{Related work}\label{sec:related work}

The question of quantitative bounds on $\mu(A_\theta)$ in the case of the almost-Mathieu operator \eqref{eqn:almostmathieu} was first investigated by B\'eguin, Valette, and Zuk \cite{beguin1997spectrum} in the context of proving mixing time bounds for random walks on the Heisenberg group; they proved  $\mu(A_\theta)=\Omega(1)$ for $\theta$ in a neighborhood of $\half$ and conjectured a constant lower bound for $\theta\in \lbr{\frac{1}{4},\half}$. This was significantly improved by Boca and Zaharescu \cite{boca2005norm} who proved nonasymptotic bounds in the full range $\theta\in \lbr{0,\half}$ via an elementary trigonometric calculation (see also Ozawa \cite{ozawa2022substitute} for a simplified proof of a slightly weaker result). Bump, Diaconis, Hicks, Miclo, and Widom \cite{bump2017exercise, bump2017useful}
gave alternative proofs of similar bounds based on an uncertainty principle for circulant matrices and by a comparison to the harmonic oscillator spectrum via a limiting argument.\footnote{This latter argument only gave asymptotic bounds as $\theta\rightarrow 0$.}  Our results do not improve the best known bound of \cite{boca2005norm}, but provide an alternate and more conceptual route to proving such bounds.

There has been an extensive study of the fine structure of the spectrum of discrete Schr\"odinger operators in the mathematical physics community, but surprisingly few works giving quantitative bounds on the edge of the spectrum. Davies \cite{davies1999ground} studied the numerical problem of estimating $\mu(A_\theta)$ and observed that it is related to subtle dynamical phenomena. Akkouche \cite{akkouche2010spectral} gave a necessary and sufficient condition for the ground state energy of a discrete Schr\"odinger operator $-L+\lambda V$ on $\Z^d$ to be positive for a large (not necessarily periodic or nonnegative) class of potentials $V$,  for small $\lambda$: $V$ merely needs to have mean bounded away from 0 on all sufficiently large intervals. This is related to the condition \eqref{eqn:adef} but holds only on some fixed macroscopic scale, whereas $a(\gt)$ is posited to exist for a chosen scale $\theta$. Milatovic \cite{milatovic2013spectral} generalized this result to other bounded degree graphs.

We will not attempt to survey the vast literature on the spectrum of Schr\"odinger operators on $\R^d$; most relevant for this paper are the works by Kato \cite{kato1952note}, Gesztesy, Graf, and Simon \cite{gesztesy1992ground}, and Arendt and Barry \cite{arendt1996spectral, arendt1997spectral}, who proved bounds on $\mu(B_\theta)$ in $d=1$ for various classes of potentials.  More broadly, there appear to be no results in the literature relating spectral bounds for discrete and continuous Schr\"odinger operators. Establishing such bounds, as is done here, could provide a fruitful bridge between the two areas.

Motivated by this paper, Becker, Wittsten, and Zworski \cite{witt} used semiclassical
methods to significantly improve some aspects of \thmref{thm:lowerbound}. They establish
\begin{equation} \label{eqn:wz}
     \frac{\mu(A_\theta)}{\mu(B_\theta)}=1+O_d(\theta) \qquad \text{as $\theta\to 0$}
\end{equation}
for all periodic smooth confining potentials $V:\R^d\rightarrow\R_{\ge0}$, with a full asymptotic expansion in $\theta$ when $V$ satisfies further nondegeneracy assumptions. Their result is stronger than \thmref{thm:lowerbound} in that it obtains a smaller ratio as a function of $\theta$, but weaker in that it does not provide explicit estimates for fixed $\theta$. The classes of potentials considered in the two theorems are  incomparable: \cite{witt} requires $V$ to be confining, whereas \thmref{thm:lowerbound} requires $V$ to be coercive with critical points of finite order. \thmref{thm:simple} proves a much weaker bound than \eqref{eqn:wz}, but requires fewer assumptions on $V$, and \thmref{thm:sampling} proves a stronger upper bound of $1$ for irrational $\theta$. 

\subsection{Preliminaries}\label{sec:prelims}
\paragraph{Notation.}  We introduce the {\em Rayleigh quotient} for any operator $T$ and function $f\in D(T)$:
$$\cR_T(f)=\frac{\<f,Tf\>}{\<f,f\>}.$$ We use $\norm{\cdot}$ to denote the $\ell_2$ and $L^2$ norms of functions, and the operator norms of operators.

We use $\E$ to denote the expectation of a random variable, and $\E_x\lbr{f(x,y)}$ to denote the partial expectation of an expression with respect to the random variable $x$ only, treating the other variables as constants.

\paragraph{$\bvep$-pseudo-ground-states.}

For $f\in D(T)$, we say that $f$ is a {\em $\eps$-pseudo-ground-state} or {\em $\eps$-p.g.s.}\ if $\cR_T(f)\le\mu(T)+\eps$. When $T$ is either $A_\gt$ or $B_\gt$ and $f$ is an $\eps$-p.g.s.\ of $T$, we may assume that $f$ is non-negative by considering $\abs{f}$, which is easily seen to have a smaller Rayleigh quotient than $f$. We will sometimes want to assume $f$ is positive; in this case we add a small everywhere-positive perturbation from $D(T)$ which increases the Rayleigh quotient by an arbitrarily small amount.

\paragraph{Fourier coefficients on a hypercube.}
We will often consider functions restricted to $C(y)$ for some choice of $y\in\gt\Z^d$; it will be useful to view these as functions defined on a hypercube. To that end, we recall some facts from Fourier analysis on a hypercube: consider a function $F$ defined on the $d$-dimensional hypercube $\{-1, 1\}^d$ with the standard edge set. The functions $\rchi_S(x) := \prod\limits_{i \in S} x_i$ for $S \subseteq [d]$ form an orthogonal basis for the space of functions on $\{-1, 1\}^d$ with respect to the inner product 
$$ \<f, g\> := \sum\limits_{x \in \{-1, 1\}^d} f(x)g(x).$$ 
We can then write $F(x) = \sum\limits_{S\subseteq[d]} \widehat{F}(S) \rchi_S(x)$. From \cite{o2014analysis} we have 
\begin{align}\label{eq: intro fourier}
    \E \lbr{F(x)^2} = \sum\limits_{S \subseteq [d]} \widehat{F}(S)^2, && \text{and} && \E \lbr{F(x)\cdot LF(x)} = 2\sum\limits_{S \subseteq [d]} \abs{S}\widehat{F}(S)^2 && \text{so} && \<F, LF \> = 2^{d+1} \sum\limits_{S \subseteq [d]} \abs{S}\widehat{F}(S)^2
\end{align}
where the expectation is taken over $x$ drawn uniformly at random from $\{-1, 1\}^d$ and $L$ is the unnormalized graph Laplacian: $Lf(x):=2df(x)-\sum\limits_{y\sim x}f(y)$. 

\subsection*{Acknowledgements}
\noindent We thank Simon Becker, Carlos Esparza, Svetlana Jitomirskaya, Fr\'ed\'eric Klopp, Jens Wittsten, and especially Maciej Zworski for many enlightening conversations. We are also grateful to an anonymous referee from {\em Pure and Applied Analysis} for many helpful comments and corrections. N.S.\ thanks Narutaka Ozawa for introducing him to this subject and many interesting discussions, and Narutaka Ozawa and Beno$\hat{\i}$t Collins for hosting him at RIMS, Kyoto University where this work was initiated. N.S.\ was partially supported by NSF grant CCF 2009011. Z.S.\ was supported by an NSF graduate research fellowship, grant number DGE 2146752. 

\section{Constructing a discrete p.g.s.\ via sampling}\label{sec:upper}
In this section we prove \thmref{thm:sampling}. 
Our strategy is: for any $g:\R^d\rightarrow\R_{\ge 0}$, we exhibit $f:\theta\Z^d\rightarrow\R_{\ge 0}$ such that $\cR_{A_\gt}(f) \le \cR_{B_\gt}(g)$. The naive construction is to let $f$ take $g$'s values on $\gt\Z^d$. The observation is that for $\gt$ irrational, ergodicity implies there is {\em a priori} no reason to prefer $\gt\Z^d\subset\R^d$ over any other shift $\gt\Z^d+\eta\subset\R^d$, and we show by averaging over all such translations that there exists a translation giving rise to a satisfactory $f$. 

\begin{proof}[Proof of {\thmref{thm:sampling}}]
Let $g:\R ^d\rightarrow\R_{\ge 0}$. Let $\eta \in [0, \theta]^d$, drawn uniformly at random with each coordinate independent. We define the test function $f_\eta$ as $f_\eta(k\gt) := g(k\theta + \eta)$.

Since we sample $g$ at $\gt \mathbb{Z}^d+\eta$ to construct $f_\eta$, it is natural to sample $V$ on the same lattice. To this end define $A_{\theta, \eta}$, a discrete operator on functions defined on $\theta\Z ^d$, via $$A_{\theta, \eta}f(k\gt) := 2d f(k\gt) - \sum\limits_{\substack{j,k\in\Z^d\\j\sim k}} f(j\gt)+\sum\limits_{k\in\Z^d}V(k\gt + \eta)f(k\gt).$$ We aim to show that there exists an $\eta$ so $\cR_{A_{\gt, \eta}}(f_\eta) \le \cR_{B_\gt}(g)$. First we have the following lemma, relating $\mu(A_{\theta, \eta})$ and $\mu(A_\theta)$, which is known (appearing in e.g.\ \cite{jito} for the almost-Mathieu operator), but we include a proof in \secref{sec:sobolev proof} for completeness.

\begin{lemma} \label{lem:shiftinv}
    If $V:\R^d\rightarrow\R_{\ge 0}$ is continuous and periodic and $\theta$ is irrational, then $\mu(A_{\theta, \eta})$ is independent of $\eta$.
\end{lemma}

It is thus sufficient to relate $\mu(A_{\gt, \eta})$ and $\mu(B_\gt)$. We will look at 
\begin{equation} \label{eq:ratioexp}
\frac{\E\lbr{\<f_\eta, A_{\gt, \eta}f_\eta \>}}{\E\lbr{\norm{f_\eta}^2}}. \end{equation} 
Note that if $X$ and $Y$ are non-negative random variables with $\E\lbr{\abs{X}} < \infty$ and $X \neq 0$ almost surely, then with nonzero probability
\begin{equation} \label{eq:probfact}
\frac{Y}{X} \le \frac{\E\lbr{Y}}{\E\lbr{X}}.\end{equation}
By \secref{sec:prelims}, we can assume that $f > 0$, so $\norm{f_\eta}^2 > 0$ almost surely and it will suffice to bound the ratio of expectations in \eqref{eq:ratioexp}.

The following segmentation will be useful in our analysis when we want to vary only one coordinate: for $k \in \Z ^d, 1 \le i \le d, \eta \in [0, \theta)^d$, let
$$I_{k, i}^\eta := \lcr{x : x_j = (k \theta + \eta)_j \hspace{1mm} \forall j \neq i, \hspace{1mm} k_i\theta + \eta_i \le x_i < (k_i + 1)\theta + \eta_i}.$$
Letting $\lcr{e_i}_{i=1}^d$ be the coordinate basis of $\R^d$, we can also express these sets as 
$I_{k,i}^\eta=(k\gt+\eta)+[0,\gt)e_i.$

\begin{figure}
\centering
\begin{minipage}{.48\textwidth}
  \centering
    \scalebox{2}{
    \begin{tikzpicture}
\coordinate (00) at (0,0);
\coordinate (01) at (0,1);
\coordinate (02) at (0,2);
\coordinate (03) at (0,3);
\coordinate (10) at (1,0);
\coordinate (11) at (1,1);
\coordinate (12) at (1,2);
\coordinate (13) at (1,3);
\coordinate (20) at (2,0);
\coordinate (21) at (2,1);
\coordinate (22) at (2,2);
\coordinate (23) at (2,3);
\coordinate (30) at (3,0);
\coordinate (31) at (3,1);
\coordinate (32) at (3,2);
\coordinate (33) at (3,3);
\coordinate (p) at (1/3, 1.7);
\coordinate (q1) at (1/3, 2.7);
\coordinate (q2) at (1+1/3, 1.7);
\coordinate (l) at (0, 1.7);
\coordinate (r) at (3, 1.7);
\coordinate (b) at (1/3, 0);
\coordinate (t) at (1/3, 3);
\foreach \c in {00,01,02,03,10,11,12,13,20,21,22,23,30,31,32,33} { \filldraw[black] (\c) circle (0.5pt); };
\draw [line width=1, blue] (p) -- (q1);
\draw [line width=1, blue, dotted] (b) -- (t);
\draw [line width=1, purple] (p) -- (q2);
\draw [line width=1, purple, dotted] (l) -- (r);
\node at (1, 1.5) {\tiny $I_{k,1}^\eta$};
\node at (2.1, 1.5) {\tiny $J_{k',1}^\eta$};
\node at (0.65, 2.2) {\tiny $I_{k,2}^\eta$};
\node at (0.65, 0.5) {\tiny $J_{k',2}^\eta$};
\node at (0,0.13) {\tiny $(0,0)$};
    \end{tikzpicture}
    }
    \caption{In $d=2$, $I_{k,i}^\eta$ (solid) and $J_{k',i}^\eta$ (dotted) for $i=1$ (in purple) and $i=2$ (in blue), $k=(0,1)$, and $\eta=\lpr{\frac{1}{3}\gt,\frac{7}{10}\gt}$.}
    \label{fig:I,J}
\end{minipage}
\quad
\begin{minipage}{.48\textwidth}
  \centering
  \scalebox{2}{
    \begin{tikzpicture}
\coordinate (00) at (0,0);
\coordinate (01) at (0,1);
\coordinate (02) at (0,2);
\coordinate (03) at (0,3);
\coordinate (10) at (1,0);
\coordinate (11) at (1,1);
\coordinate (12) at (1,2);
\coordinate (13) at (1,3);
\coordinate (20) at (2,0);
\coordinate (21) at (2,1);
\coordinate (22) at (2,2);
\coordinate (23) at (2,3);
\coordinate (30) at (3,0);
\coordinate (31) at (3,1);
\coordinate (32) at (3,2);
\coordinate (33) at (3,3);
\coordinate (p) at (1/3, 1.7);
\coordinate (q1) at (1/3, 2.7);
\coordinate (q2) at (1+1/3, 1.7);
\coordinate (l) at (0, 1.7);
\coordinate (r) at (3, 1.7);
\coordinate (b) at (1/3, 0);
\coordinate (t) at (1/3, 3);
\draw[draw=pink, opacity=0, fill=pink, fill opacity=1] (01) -- (31) -- (32) -- (02) -- (01);
\draw [line width=1, purple, dotted] (l) -- (r);
\node at (1.5, 1.5) {\tiny $J_{k',1}^\eta$};
\node at (0,0.13) {\tiny $(0,0)$};
\foreach \c in {00,01,02,03,10,11,12,13,20,21,22,23,30,31,32,33} { \filldraw[black] (\c) circle (0.5pt); };
    \end{tikzpicture}
    }
    \caption{In $d=2$, $J_{k',1}^\eta$ (purple, dotted) and $\bigsqcup_{0\le\gw_2<\gt}J_{k',1}^{\gw}$ (pink) for $k'=k_2=1$. Note that this union is independent of $\eta=\lpr{\frac{1}{3}\gt,\frac{7}{10}\gt}$.}
    \label{fig:J union}
\end{minipage}
\end{figure}

Given $k$ and $i$, let $k_{\neg i}:=\lpr{k_1,\dots,k_{i-1},k_{i+1},\dots,k_d} \in \Z^{d-1}$. For $k' \in \Z^{d-1}$  define 
\begin{align*}
    J_{k', i}^\eta &:=\bigcup\limits_{\substack{k \in \Z^d \\ k_{\neg i} = k'}}I_{k,i}^\eta.
\end{align*} 
Immediately from this we have \begin{align} \label{eq:unionj}
    \bigsqcup_{\substack{0\le\eta_j<\gt \\ j \neq i}} J_{k', i}^{\eta} = \bigcup_{\substack{k \in \Z^d \\ k_{\neg i} = k'}} \ol{C}(\gt k), && \text{therefore} && \bigsqcup_{k' \in \Z^{d-1}} \bigsqcup_{\substack{0\le\eta_j<\gt \\ j \neq i}} J_{k',i}^{\eta} = \R^d.
\end{align} 
Note $J_{k',i}^\eta$ is independent of $\eta_i$ but is dependent on each $\eta_j$, $j\neq i$. See \figref{fig:I,J} for examples of $I_{k,i}^\eta$ and $J_{k',i}^\eta$ and \figref{fig:J union} for an illustration of \eqref{eq:unionj}. 

We evaluate each quantity involved for the main result.
\begin{claim}
    The norms of $f_\eta$ and $g$ satisfy $\E _\eta\lbr{\norm{f_\eta}^2} = \frac{1}{\theta^d}\norm{g}^2.$
\end{claim}
\begin{proof} This follows from directly evaluating this expectation:
    \begin{align*}
        \E_\eta\lbr{\norm{f_\eta}^2} = \E_\eta\lbr{\sum\limits_{k \in \Z ^d} f_\eta(k\gt)^2} = \sum\limits_{k \in \Z ^d} \E_\eta\lbr{g(k\theta + \eta)^2} = \sum\limits_{k \in \Z ^d} \E _{\eta_1}\lbr{\cdots \E_{\eta_d}\lbr{g(k\theta + \eta)^2}\cdots} = \frac{1}{\theta^d}\norm{g}^2.
    \end{align*}
\end{proof}

\begin{claim}
    The quadratic forms of $f_\eta$ and $g$ with the respective Laplacian operators satisfy $$\E_\eta\lbr{\< f_\eta, L f_\eta \>} \le \frac{1}{\theta^{d-2}} \norm{\nabla g}^2.$$
\end{claim}
\begin{proof}
    First we evaluate $\< f_\eta, L f_\eta \>$: \begin{align*}
        \< f_\eta, L f_\eta \> &= \sum\limits_{k \in \Z ^d} \sum\limits_{i =1}^d \lpr{f_\eta((k+e_i)\gt) - f_\eta(k\gt)}^2 \\
        &= \sum\limits_{i = 1}^d \sum\limits_{k \in \Z ^d} \lpr{g((k+e_i)\theta + \eta) - g(k\theta + \eta)}^2 && \text{construction of $f_\eta$}\\
        & = \sum\limits_{i=1}^d \sum\limits_{k \in \Z ^d} \lpr{\int_{I_{k,i}^\eta} \frac{\partial}{\partial x_i} g(x) \d x}^2 &&\\
        &\le \sum\limits_{i = 1}^d \sum\limits_{k \in \Z ^d} \theta \int_{I_{k, i}^\eta} \lpr{\frac{\partial}{\partial x_i} g(x)}^2 \d x && \text{Cauchy--Schwarz.}
    \end{align*}
    Then, 
    \begin{align*}
        \E_\eta\lbr{\< f_\eta, Lf_\eta \>} &\le \theta \sum_{i = 1}^d \sum_{k' \in \Z^{d-1}} \E_\eta\lbr{\int_{J_{k',i}^\eta} \lpr{\frac{\partial}{\partial x_i} g(x)}^2 \d x} \\
        & = \theta \sum_{i = 1}^d \E_\eta\lbr{\sum_{k' \in \Z^{d-1}} \int_{J_{k',i}^\eta} \lpr{\frac{\partial}{\partial x_i} g(x)}^2 \d x} && \text{Tonelli} \\
        & = \theta \sum_{i = 1}^d \E_{\lpr{\eta_1,\dots,\eta_{i-1},\eta_{i+1},\dots,\eta_d}}\lbr{\sum_{k' \in \Z^{d-1}} \int_{J_{k',i}^\eta} \lpr{\frac{\partial}{\partial x_i} g(x)}^2 \d x} && \text{invariance of $J_{k',i}^\eta$ w.r.t.\ $\eta_i$}\\
        & = \frac{\theta}{\theta^{d-1}} \sum_{i = 1}^d \int_{\R ^d} \lpr{\frac{\partial}{\partial x_i} g(x)}^2 \d x  && \text{from \eqref{eq:unionj}}\\
        & = \frac{1}{\theta^{d-2}}\norm{\nabla g}^2.
    \end{align*}
\end{proof}

\begin{claim}
    The quadratic forms of $f_\eta$ and $g$ with the respective potential operators satisfy $$\E _\eta\lbr{\<f_\eta, V f_\eta \>} = \frac{1}{\theta^d} \< g, Vg \>.$$
\end{claim}
\begin{proof}
    Again, we directly evaluate this expectation \begin{align*}
        \E\lbr{\< f_\eta, V f_\eta \>} = \E_\eta\lbr{\sum\limits_{k \in \Z ^d} V(k\theta + \eta)f_\eta(k\gt)^2} = \sum\limits_{k \in \Z ^d} \E_\eta\lbr{V(k\theta + \eta) g(k\theta + \eta)^2} = \frac{1}{\theta^d} \< g, Vg \>.
    \end{align*}
\end{proof}

Altogether then we see
$$\frac{\E_\eta\lbr{\< f_\eta, A_{\gt,\eta} f_\eta \>} }{\E_\eta\lbr{\< f_\eta, f_\eta \>}} \le \lpr{\frac{1}{\theta^{d-2}} \norm{\nabla g}^2 + \frac{1}{\theta^d} \< g, Vg \> } \frac{\theta^d}{\norm{g}^2} = \frac{\theta^2 \norm{\nabla g}^2 + \< g, Vg \>}{\norm{g}^2}.$$

Minimizing over choices of $g$, by \eqref{eq:probfact} there is $\eta$ such that $\mu(A_{\gt,\eta}) \le \mu(B_\theta)$ for irrational $\gt$. Finally by \lemref{lem:shiftinv}, $\mu(A_\gt) \le \mu(B_\gt)$.
\end{proof}

The above proof considers the infimum of the spectrum of the operator $A_\gt$. However, we can instead consider its {\em union spectrum}, which is more natural in some contexts. 
\begin{definition}[Union Spectrum, cf.\ {\cite{marx2017dynamics}}] \label{def:union}
    The {\em union spectrum} of $A_\gt$ is $S_+(A_\gt) := \bigcup\limits_\eta \sigma(A_{\gt, \eta})$.
\end{definition}

Note that if $\gt$ is irrational, then by \lemref{lem:shiftinv}, $\inf S_+(A_\gt) = \mu(A_\gt)$. By considering the infimum of the union spectrum, we can extend \thmref{thm:sampling} to all $\gt$: 

\begin{corollary}\label{cor:1d union spec}
    For all $\gt>0$, $\inf S_+(A_\gt) \le \mu(B_\gt)$.
\end{corollary}
\begin{proof}
    By the above proof we know that there exists $\eta$ such that$$\mu(A_{\gt,\eta}) \le \frac{\< f_{\eta}, A_{\theta,\eta} f_{\eta} \>}{\<f_\eta, f_\eta \>} \le \frac{\<g, B_\theta g \>}{\<g, g \>}.$$ Minimizing over $g$ and recalling the definition of $S_+(A_\gt)$, the claim is shown.
\end{proof}

For choices of $V$ such that $\mu(A_\gt)$ is continuous with respect to $\gt$, this result extends immediately to rational $\gt$ by continuity.

\begin{corollary}\label{corr:continuous}
    For $V$ such that $\mu(A_\gt)$ is continuous with respect to $\gt$, $\mu(A_\gt) \le \mu(B_\gt)$ for all $\gt$.
\end{corollary} 

\begin{remark} \label{rem:irrat}
    Choices of $V$ satisfying \corref{corr:continuous} are not immediately apparent. A notable counterexample is the almost-Mathieu operator. Let $V_\eta(x) = 2-2\cos(2\pi x + \eta)$ and $V = V_0$. Then for $\gt = 1$ and any $\eta$, we have $V_\eta(k\gt) = V(\eta)$ for all $k$, so $\mu(A_{\gt, \eta}) = V(\eta)$. Now consider $\gt' = 1 - \eps$ where $\eps$ is a small positive irrational (so $\gt'$ is irrational). Then for $\eta = 0$, we have $V(k\gt') = V(k(1-\eps)) = V(-k\eps) = V(k\eps)$, so $\mu(A_{1-\eps}) = \mu(A_{\eps})$. By \propref{prop:muBlb intro} and \thmref{thm:sampling}, $\mu(A_{\eps}) \le O(\eps)$, so as $\eps \rightarrow 0$ along a sequence of irrationals, $\mu(A_{\gt'}) \rightarrow 0$. Since $\gt'$ is irrational, $\mu(A_{\gt', \eta}) \rightarrow 0$ for all $\eta$, so we see $\mu(A_{\gt, \eta})$ is not continuous in $\gt$ at $\gt = 1$ for any $\eta$ such that $V(\eta) > 0$ (as $\mu(A_{1,\eta})=V(\eta)$). By considering e.g.\ $2-2\cos(2\pi kx)$ for integers $k\ge 1$, one can also create discontinuities at $\theta=\frac{1}{k}$.
\end{remark}

\section{Constructing a continuous p.g.s.\ via multilinear interpolation}\label{sec:lowerasymp}

In this section we prove \thmref{thm:lowerbound}. Whereas the previous direction of the proof used the probabilistic method to find a pseudo-ground-state for $A_\gt$ from one for $B_\gt$, here we will deterministically construct a pseudo-ground-state for $B_\gt$ from one for $A_\gt$ by multilinear interpolation. In $d=1$ this amounts just to ``connecting'' the function values by line segments. In higher dimensions, on each hypercube in $\gt\Z^d$, the interpolating function is the unique multilinear function in each variable which equals the given function on the corners of the hypercube. The techniques brought to bear are first Boolean Fourier-analytic, in terms of understanding both local and global behavior of the functions, and then a careful consideration of the potential's polynomial behavior near its minimizers. 

The precise result that we prove, an explicit form of \thmref{thm:lowerbound}, is: 
\begin{theorem} \label{thm:viataylor}    Let $V:\mathbb{R^d} \rightarrow \mathbb{R}_{\ge 0}$ be $\mathbb{Z}^d$-periodic, sufficiently smooth, and $(t_0, K, P)$ coercive. Let $L$ be the Lipschitz constant of $V$ and let $\frac{D_v^kD_uV}{k!} \le M_P$ for any unit vectors $u, v,$ and $k \le P-1$ where $D_v$ is the directional derivative with respect to $v$. For any $q \in (0,1)$ and and $\gt$ small enough that $\gt^q \le t_0$, 
    $$\lpr{1+\frac{1}{1-\mu(A_\gt)}\lpr{\mu(A_\gt)+\frac{2L\gt^{1-q}}{1-L\gt^{1-q}}}}\mu(A_\gt) \ge \mu(B_\gt) - \frac{1}{1-\mu(A_\gt)} \lpr{2M_PK\sqrt{d}\gt^{1+q(P-1)/P}}.$$
\end{theorem}

\subsection{Defining the test function}\label{sec:def test fn}

First we set up some notation and basic facts. For $y\in\gt\Z^d$, $C(y)$ identifies naturally with the vertices of the hypercube $\{-1,1\}^d$, and we let $E(C(y))$ be the edge set of the corresponding graph. 

Let $f$ be a positive bounded function on $\gt \mathbb{Z}^d$. Let $g$ be the piecewise multilinear extension of $f$, defined on $\R ^d$ as 
\begin{equation}
g(x_1, \dots, x_d) := \E\limits_{x_1, \dots, x_d}\lbr{f(k_{x_1}\theta, \dots, k_{x_d}\theta)}\label{eq:extension def}
\end{equation}
where each $k_{x_i}$ is independent with the following distribution for $k_i\theta \le x_i \le (k_i+1)\theta$: \begin{align*}
    \bP(k_{x_i} = k_i) &= \frac{(k_i+1)\theta - x_i}{\theta}, \\
    \bP(k_{x_i} = k_i+1) &= \frac{x_i - k_i\theta}{\theta}  = 1 - \bP(k_{x_i} = k_i).
\end{align*}

The function $g$ obtained in this way is continuous on each closed hypercube, and is continuous on all of $\R^d$ since $g$ is defined uniquely on the intersection of any two adjacent cubes: if $x_i=k_i\gt$ then $\bP(k_{x_i}=k_i-1)=\bP(k_{x_i}=k_i+1)=0$ and so the expectation is taken over the same set of lattice points, namely those forming the codimension 1 hypercube of the intersection. 

It will be convenient to express $g$ in terms of the Boolean Fourier coefficients of $f$ on each hypercube $C(y)$. To that end, given $y\in \theta \Z^d$, let $v_y$ be a vector such that $x \mapsto \frac{\gt}{2}x - v_y$ sends $C(y) \to \{-1, 1\}^d$ and $y \mapsto \mathbf{-1}$. Define the following functions: \begin{align*}
    F(x) := f\lpr{\frac{\gt}{2}x}, \text{ a dilation of $f$}, && \text{and} && F_y(x) := f\lpr{\frac{\gt}{2}x - v_y}, \text{ a dilation and translation of $f$.}
\end{align*}
Note that $F$ is independent of $y$, and $F_y$ is simply a shifted version of $F$. We will only ever consider $F_y(x)$ for $x \in \{-1, 1\}^d$, so we restrict it to this domain. As introduced in \secref{sec:prelims} we can write $F_y$ as
$$F_y(x) = \sum\limits_{S \subseteq [d]} \widehat{F}_y(S) \rchi_S(x).$$

Similarly, we define $G(x) := g\lpr{\frac{\gt}{2}x}$ and $G_y := g\lpr{\frac{\gt}{2}x - v_y}$. Since there is a unique multilinear function agreeing with $F_y$ on $\{-1,1\}^d$, we have the expansion $$G_y(x) = \sum\limits_{S \subseteq [d]} \widehat{F}_y(S) \rchi_S(x)$$ for $x \in [-1, 1]^d$.

 It will be convenient to relate $\E\lbr{G_y(x)^2}$ (where the expectation is taken over $x$ drawn uniformly at random from $[-1,1]^d$) and $\<G, -\Delta G\>$ to $\lcr{\widehat{F}_y(S)}$. This simply amounts to integrating the characters with respect to the Lebesgue measure and differentiating. For any $\rchi_S$, we have
$$\int_{[-1,1]^d} (\rchi_S(x))^2 \d x = \int_{[-1, 1]^d} \prod_{i \in S} x_i^2 \d x = 2^{d - \abs{S}} \prod_{i \in S} \int_{-1}^1 x_i^2 \d x_i = 2^{d - \abs{S}} \lpr{\frac{2}{3}}^{\abs{S}} = 2^d \lpr{\frac{1}{3}}^{\abs{S}}.$$
We then have for every $y\in \theta \Z^d$:
\begin{align}
    \E \lbr{G_y(x)^2} &= \sum\limits_{S \subseteq [d]} \lpr{\frac{1}{3}}^{\abs{S}}\widehat{F_y}(S)^2, && \text{and} \label{eq: intro fourier cont 1} \\ \E \lbr{\lpr{\frac{\partial}{\partial x_i} G_y(x)}^2} &= \sum\limits_{S:  i \in S} \lpr{\frac{1}{3}}^{\abs{S}-1}\widehat{F_y}(S)^2 && \text{so} && \<G_y, -\Delta G_y \> = 2^{d} \sum\limits_{S \subseteq [d]} \lpr{\frac{1}{3}}^{\abs{S}-1}\abs{S}\widehat{F_y}(S)^2 \label{eq: intro fourier cont 2}.
\end{align}

To prove \thmref{thm:viataylor}, we will analyze the contributions from the Laplacian terms and the potential terms to the Rayleigh quotient separately. 
\subsection{Laplacian term analysis}
We want to find expressions for $\norm{f}^2, \norm{g}^2, \<f, Lf\>,$ and $\<g, -\Delta g\>$ with respect to $\widehat{F}_y$ for $y\in\gt\Z^d$. Since we won't be looking at a single cube in these settings, we define 
$$m_k = \sum\limits_{S: \abs{S} = k} \sum\limits_{y\in\Z^d} \widehat{F}_y(S)^2.$$
\begin{proposition}\label{prop:fourierlap}
    For $f, g$ as above,
    \begin{align}
    \norm{f}^2 = \sum\limits_{k=0}^d m_k, && \<f, Lf\> = 4 \sum_{k=0}^d km_k, && \norm{g}^2 = \gt^d \sum_{k=0}^d \lpr{\frac{1}{3}}^k m_k, && \<g, -\Delta g\> = 4 \gt^{d-2} \sum_{k=0}^d \lpr{\frac{1}{3}}^{k-1}km_k.\label{eq:fourier tech results}
\end{align}
\end{proposition}

\begin{proof} Recall from \eqref{eq: intro fourier} that for every $y\in \theta \Z^d$,
\begin{align*}
    \E \lbr{F_y(x)^2} = \sum\limits_{S \subseteq [d]} \widehat{F}_y(S)^2 && \text{and} && \<F_y, LF_y \> = 2^{d+1} \sum\limits_{S \subseteq [d]} \abs{S}\widehat{F}_y(S)^2.
\end{align*}
We relate the norm and Laplacian terms on a single cube to those on the whole space:
\begin{align*}
    \norm{f}^2 = \frac{1}{2^d} \sum_y \norm{F_y}^2 && \text{and} && \<f, Lf \> = \frac{1}{2^{d-1}} \sum_y \<F_y, LF_y\>.
\end{align*}
Combining these we have
\begin{align*}
    \norm{f}^2 = \sum\limits_{k=0}^d m_k && \text{and} && \<f, Lf\> = 4 \sum_{k=0}^d km_k,
\end{align*}
as desired. Recall from \eqref{eq: intro fourier cont 1} and \eqref{eq: intro fourier cont 2} that
\begin{align*}
    \E \lbr{G_y(x)^2} = \sum\limits_{S \subseteq [d]} \lpr{\frac{1}{3}}^{\abs{S}}\widehat{F}_y(S)^2 && \text{and} && \E \lbr{\lpr{\frac{\partial}{\partial x_i} G(x)}^2} = \sum\limits_{S:  i \in S} \lpr{\frac{1}{3}}^{\abs{S}-1}\widehat{F}(S)^2.
\end{align*}
Taking into account the dilation and the change of variables, we see that
\begin{align*}
    \norm{g}^2 = \lpr{\frac{\gt}{2}}^d \sum_y \norm{G_y}^2 && \text{and} && \<g, -\Delta g \> = \frac{2^2}{\gt^2} \lpr{\frac{\gt}{2}}^d \<G, -\Delta G\>.
\end{align*}
Combining these we have
\begin{align*}
    \norm{g}^2 = \gt^d \sum_{k=0}^d \lpr{\frac{1}{3}}^k m_k && \text{and} && \<g, -\Delta g\> = 4 \gt^{d-2} \sum_{k=0}^d \lpr{\frac{1}{3}}^{k-1}km_k,
\end{align*}
as desired. 
\end{proof}

Given \eqref{eq:fourier tech results}, we now prove the following.
\begin{lemma}\label{lem:interp lap ub}
    Let $f$ be an $\eps$-pseudo-ground-state for $A_\gt$ and let $g$ be its multilinear extension. Then $$ \frac{\cR _{-\gt^2\Delta}(g)}{\cR _{L}(f)} \le \frac{1}{1-\mu(A_\gt)-\eps}.$$
\end{lemma}
\begin{proof}
    In the following, keep in mind that $m_k$ is dependent on $\gt$. Assume without loss of generality that $\norm{f}^2 = \sum\limits_{k=0}^d m_k = 1$. From \eqref{eq:fourier tech results} we have \begin{align*}
    &\cR_{-\gt^2\Delta}(g) = -\gt^2\frac{\<g, \Delta g\>}{\norm{g}^2} = \frac{4 \sum_{k=0}^d \lpr{\frac{1}{3}}^{k-1}km_k}{\sum_{k=0}^d \lpr{\frac{1}{3}}^k  m_k} & \text{and} &
    &4\sum\limits_{k=0}^d \lpr{\frac{1}{3}}^{k-1} km_k \le 4\sum\limits_{k=0}^d km_k = \<f, Lf\> 
    \end{align*}
    so that
    \begin{align*}
    &\cR_{-\gt^2\Delta}(g) \le \frac{\<f, Lf\>}{\sum_{k=0}^d \lpr{\frac{1}{3}}^k  m_k} = \frac{\cR _{L}(f)}{\sum_{k=0}^d \lpr{\frac{1}{3}}^k  m_k}.
    \end{align*}
    We also note that if $d = 1$, this becomes an equality.

    Now we want to lower bound $\sum\limits_{k=0}^d \lpr{\frac{1}{3}}^k  m_k$. It will suffice to just lower bound $m_0$. Since $\sum\limits_{k=0}^d m_k = 1$, we can view $\<f, Lf\>$ as a weighted average of elements of $[d]$ (times 4). Since $f$ is an $\eps$-p.g.s., we know $\cR _{A_\gt}(f) \le \mu(A_\gt) + \eps$. From this and \eqref{eq:fourier tech results} we must have $m_0 \ge 1-(\mu(A_\gt)+\eps) = 1-\mu(A_\gt)-\eps$, so
    $$\sum\limits_{k=0}^d \lpr{\frac{1}{3}}^{k}m_k = m_0 + \frac{1}{3}m_1 + \dots \ge 1 - \mu(A_\gt)-\eps. $$

    Altogether we have 
    $$\frac{\cR _{-\gt^2\Delta}(g)}{\cR _{L}(f)} \le \frac{1}{\sum_{k=0}^d \lpr{\frac{1}{3}}^k m_k} \le \frac{1}{1-\mu(A_\gt) - \eps}$$
    as desired. 
\end{proof}

\subsection{Potential term analysis}
Now all that remains is to analyze the potential terms. We will analyze the Rayleigh quotient of the potential terms restricted to each $C(y)$ (or $\ol{C}(y)$ for the continuous case). We define the following quantities:
\begin{align*}
    \norm{g}^2_{\ol{C}(y)} &:= \int_{\ol{C}(y)} g(x)^2 \d x, & \<g, Vg\>_{\ol{C}(y)} &:= \int_{\ol{C}(y)} V(x) g(x)^2 \d x, \\
    \norm{f}^2_{C(y)} &:= \sum\limits_{x \in C(y)} f(x)^2, & \<f, Vf\>_{C(y)} &:= \sum_{x \in C(y)}V(x)f(x)^2.
\end{align*}

If we do a ``cube by cube'' analysis, we then need some way to relate this to the whole Rayleigh quotient. Observe that we can relate the ``local'' Rayleigh quotients to the ``global'' Rayleigh quotients via 
\begin{align}
    \frac{\< g, Vg \>}{\norm{g}^2} &= \sum_{y\in\gt\Z^d}\frac{\norm{g}^2_{\ol C(y)}}{\norm{g}^2} \frac{\<g, Vg \>_{\ol C(y)}}{\norm{g}^2_{\ol C(y)}} & \text{ and } &&
    \frac{\< f, V f \>}{\norm{f}^2} &= \sum_{y\in\gt\Z^d} \frac{\norm{f}^2_{C(y)}}{2^d\norm{f}^2} \frac{\<f, V f \>_{C(y)}}{\norm{f}^2_{C(y)}}\label{eq:pot rays}
\end{align} 
where the $2^d$ factor corrects for overcounting, since each vertex $x$ is contained in $2^d$ cubes. It is natural to define the following probability measures on $\gt\mathbb{Z}^d$:
\begin{align}
    \nu_g(y) &:= \frac{\norm{g}^2_{\ol C(y)}}{\norm{g}^2} & 
    \text{ and } && \nu_f(y) &:= \frac{\norm{f}^2_{C(y)}}{2^d\norm{f}^2} = \frac{\E\limits_{k\in C(y)}f(k)^2}{\norm{f}^2}.\label{eq:nu defs}
\end{align}

\begin{lemma}\label{clm:m meas rel}
    For $f$ and $g$ as above, for all $y\in\gt\Z^d$, 
    $$\nu_g(y)\le\frac{1}{1-\mu(A_\gt)-\eps}\nu_f(y).$$
\end{lemma}
\begin{proof}
    Assume $\norm{f}^2 = \sum\limits_{k=0}^d m_k = 1$. Recall from \eqref{eq: intro fourier} and \eqref{eq:fourier tech results} that 
    \begin{align*}
        \norm{g}^2 = \gt^d \sum_{k=0}^d \lpr{\frac{1}{3}}^k m_k && \text{and} && \norm{g}^2_{\ol{C}(y)} = \gt^d \sum\limits_{S \subseteq [d]} \lpr{\frac{1}{3}}^{\abs{S}}\widehat{F}_y(S)^2.
    \end{align*}
    Then we have $$\nu_g(y) = \frac{\sum_{S \subseteq [d]} \lpr{\frac{1}{3}}^{\abs{S}}\widehat{F}_y(S)^2}{\sum_{k=0}^d \lpr{\frac{1}{3}}^k m_k}.$$  
    
    From \lemref{lem:interp lap ub} we must have $m_0 \ge 1-\mu(A_\gt)-\eps$. We see $\sum_{k=0}^d \lpr{\frac{1}{3}}^k m_k \ge m_0 \ge 1 - \mu(A_\gt)-\eps$, so $$\nu_g(y) \le \frac{1}{1-\mu(A_\gt)-\eps} \lpr{\sum\limits_{S \subseteq [d]} \lpr{\frac{1}{3}}^{\abs{S}}\widehat{F}_y(S)^2}.$$
    
    From \eqref{eq: intro fourier}, we can write $\nu_f(y)$ as 
    $$\nu_f(y) = \sum\limits_{S \subseteq [d]} \widehat{F}_y(S)^2.$$
    Clearly $\sum\limits_{S \subseteq [d]} \lpr{\frac{1}{3}}^{\abs{S}}\widehat{F}_y(S)^2 \le \sum\limits_{S \subseteq [d]} \widehat{F}_y(S)^2$, so the claim is shown.
\end{proof}

Now we arrive at the main estimate. 

\begin{lemma}\label{lem:viataylor}
    Let $V:\mathbb{R^d} \rightarrow \mathbb{R}_{\ge 0}$ be smooth, $\mathbb{Z}^d$-periodic and $(t_0, K, P)$-coercive. Let $L$ be the Lipschitz constant of $V$ and let $\frac{D_v^kD_uV}{k!} \le M_P$ for any unit vectors $u, v,$ and $k \le P-1$. For any $q \in (0,1)$ and $\gt$ small enough that $\gt^q \le t_0$,
    $$\abs{\cR_{V}(g)-\cR_{V}(f)} \le \frac{1}{1-\mu(A_\gt)-\eps} \lpr{2M_PK\sqrt{d}\gt^{1+q(P-1)/P} + \frac{2L\gt^{1-q}}{1-L\gt^{1-q}}\cR_{V}(f)}.$$
\end{lemma}
\begin{proof}

\noindent Suppose $y\in \theta \Z^d$. Let $q \in (0,1)$ and $\gt$ small enough that $\gt^q \le t_0$.

\smallskip

\noindent {\em Case 1.} $V(y)\ge\gt^q$. Then, for $x\in\ol{C}(y)$, 
$$1-L\gt^{1-q}\le\frac{V(x)}{V(y)}\le1+L\gt^{1-q},$$
so that we have
\begin{align*}
    \cR_{V,\ol{C}(y)}(g)&\le V(y)\lpr{1+L\gt^{1-q}},\\
    \cR_{V,C(y)}(f)&\ge V(y)\lpr{1-L\gt^{1-q}},\\
    \frac{\cR_{V,\ol{C}(y)}(g)}{\cR_{V,C(y)}(f)}&\le\frac{1+L\gt^{1-q}}{1-L\gt^{1-q}},\\
    \abs{\cR_{V,\ol{C}(y)}(g)-\cR_{V,C(y)}(f)}&\le\frac{2L\gt^{1-q}}{1-L\gt^{1-q}}\cR_{V,C(y)}(f).
\end{align*}

\noindent {\em Case 2.} $V(y)\le \theta^q$. 

Consider the difference $\abs{\cR_{V,\ol{C}(y)}(g)-\cR_{V,C(y)}(f)}$. Being the difference of two random variables supported in $\ol{C}(y)$, it is at most 
$$\sup\limits_{x\in\ol{C}(y)}V(x)-\inf\limits_{x\in\ol{C}(y)}V(x)\le2\sup\limits_{x\in\ol{C}(y)}\abs{V(x)-V(y)}.$$
By the fundamental theorem of calculus, 
$$\sup_{x \in \ol{C}(y)} \abs{V(x) - V(y)} \le \sqrt{d}\gt \cdot \sup_{\substack{z \in \ol{C}(y)\\u\in \bS^{d-1}}} D_uV(z).$$ 
Let $z$ and $u$ maximize this expression. 
By coercivity, $m$ is a critical point of order $p\le P$. We know $D_{z-m}^kD_uV(m) = 0$ for any $k < p-1$. By Taylor's remainder theorem we can bound 
$$D_uV(z) \le \sup_{\norm{\xi-m}_\infty<\sqrt{d}\gt}D_{z-m}^{p-1}D_uV(\xi) \cdot \frac{\norm{z-m}^{p-1}_2}{(p-1)!}.$$

By coercivity, we know that $\norm{z-m}^{p-1}_2 \le K\gt^{q(p-1)/p}$. Using the uniform bound $M_P$ we have
\begin{equation} \label{eq:localpotbound}
    \max_{x \in \ol{C}(y)} \abs{V(x) - V(y)} \le M_PK\sqrt{d}\gt \cdot \gt^{q(p-1)/p} = M_PK\sqrt{d} \gt^{1+q(p-1)/p} \le M_PK\sqrt{d} \gt^{1+q(P-1)/P}.
\end{equation}

To relate the global Rayleigh quotients we recall \eqref{eq:pot rays}. By \lemref{clm:m meas rel}, $\nu_g\le\frac{1}{1-\mu(A_\gt)-\eps}\nu_f$, so by \eqref{eq:localpotbound} we have in both cases:
\begin{equation} \label{eq:finalpot} \abs{\cR_{V}(g)-\cR_{V}(f)} \le \frac{1}{1-\mu(A_\gt)-\eps} \lpr{2M_PK\sqrt{d}\gt^{1+q(P-1)/P} + \frac{2L\gt^{1-q}}{1-L\gt^{1-q}}\cR_{V}(f)}.\end{equation}

\end{proof}

\subsection{Proof of {\thmref{thm:viataylor}}}

\begin{proof}[Proof of {\thmref{thm:viataylor}}] 

We first verify that the $g$ constructed in \eqref{eq:extension def} is a valid test function for \eqref{eqn:bvar}. Note that $g$ is continuous and equal to a polynomial of degree at most $d$ on each cube $\overline{C}(y)$. Since $f\in L^\infty(\R^d)$, $g$ has all derivatives $\partial_i g$ uniformly bounded on $\R^d=\bigcup\limits_{y\in \theta \Z^d}\mathrm{int}(\overline{C}(y))$. Thus, $g$ is Lipschitz. By e.g.\ \cite[pp.\ 279]{evans2022partial}, we see that $g\in W^{1,\infty}(\R^d)$. By Jensen's inequality we further have $\norm{g}_{\overline{C}(y)}^2\le O\lpr{\norm{f}^2_{C(y)}}$ for every $y$, so $g\in H^1(\R^d)$, as desired.

We now aim to bound the difference between $$\cR_{B_\gt}(g) = \cR_{-\gt^2\Delta}(g) + \cR_{V}(g) \hspace{5mm} \text{and} \hspace{5mm} \cR_{A_\gt}(f) = \cR_{L}(f) + \cR_{V}(f).$$
From \lemref{lem:interp lap ub}, multiplying through and rearranging yields the following relation on the Laplacian terms: 
$$\abs{\cR_{-\gt^2\Delta}(g) - \cR _{L}(f)} \le \frac{\mu(A_\gt)+\eps}{1-\mu(A_\gt)-\eps}\cR_{L}(f). $$

Combining this with \eqref{eq:finalpot} we have \begin{equation}
    \abs{\cR_{B_\gt}(g) - \cR_{A_\gt}(f)} \le \frac{1}{1-\mu(A_\gt)-\eps} \lpr{2M_PK\sqrt{d}\gt^{1+q(P-1)/P} + \lpr{\mu(A_\gt)+\eps+\frac{2L\gt^{1-q}}{1-L\gt^{1-q}}}\cR_{A_\gt}(f)}.\label{eq:add-mult bound}
\end{equation}

Minimizing over all $f$ we have the following bound on $\mu(B_\gt)$:
\begin{equation}
    \mu(B_\gt) \le \mu(A_\gt) + \frac{1}{1-\mu(A_\gt)} \lpr{2M_PK\sqrt{d}\gt^{1+q(P-1)/P} + \lpr{\mu(A_\gt)+\frac{2L\gt^{1-q}}{1-L\gt^{1-q}}}\mu(A_\gt)}.
\end{equation}

\end{proof}
\begin{remark} It is clear from the proof that \thmref{thm:lowerbound} holds for any periodic $V\in C^P$, not necessarily $C^\infty$.\end{remark}
\section{A bound for any potential}\label{sec:lowerany}
We give an upper bound on $\mu(B_\gt)$ with respect to $\mu(A_\gt)$ with few constraints on our potential $V$. The result we prove is the following, a slightly sharper version of \thmref{thm:simple}. 

\begin{theorem}\label{thm:exp bd}
    Let $V: \R ^d \rightarrow \R _{\ge 0}$, not necessarily periodic with $\norm{V}_\infty\le1$, and let $\gt\in(0,1)$. Suppose that there exists $a \in \mathbb{R}$ such that for all $y\in\gt\Z^d$, 
    $$\frac{1}{\gt^d}\int_{\ol{C}(y)}V(x)^2 \d x \le a \cdot \frac{1}{2^d} \sum\limits_{k \in C(y)}V(k)^2.$$
    Then, $$c_d\mu(A_\theta)\ge\mu(B_\theta)$$
    where $c_d = \max \lcr{\frac{3^d}{2^d} + 5 \cdot 3^d, \frac{3^d}{2^d} \frac{a}{r}}$ and $r$ is a dimension-dependent constant between $\frac{1}{4}$ and $1$.
\end{theorem} 

\begin{proof}
Let $f$ be a positive function on $\theta \Z ^d$. We again consider the multilinear extension $g$ as an expectation as defined in \eqref{eq:extension def}. We relate each term in the Rayleigh quotients. The following bound achieves a better constant than \propref{prop:fourierlap}.
\begin{lemma}\label{lem:norm comp}
    The norms of $f$ and $g$ satisfy $\norm{g}^2 \ge \lpr{\frac{2}{3}}^d \gt^d \norm{f}^2$.
\end{lemma}
\begin{proof} We prove this via induction on $d$. For $d = 1$ we have \begin{align*}
        \int_{\R}g(x)^2 \d x &= \sum_{k \in \Z} \int_{k\gt}^{(k+1)\gt} \frac{1}{\gt^2}\lpr{(x-k\gt)f(k\gt) + ((k+1)\gt - x)f((k+1)\gt)}^2 \d x \\
        &\ge \sum_{k \in \Z} \frac{1}{\gt^2} \int_{k\gt}^{(k+1)\gt} (x-k\gt)^2f(k\gt)^2 + ((k+1)\gt-x)^2f((k+1)\gt)^2 \d x \\
        &= \frac{\gt}{3} \sum_{k \in \Z} (f(k\gt)^2 + f((k+1)\gt)^2) \\
        &= \frac{2\gt}{3} \norm{f}^2.
    \end{align*}

    Assume that for any function $f$ defined on $\gt\Z^{d}$, its multilinear extension $g$ satisfies $\norm{g}^2 \ge \lpr{\frac{2}{3}}^{d} \gt^{d} \norm{f}^2.$

    Consider $f$ defined on $\gt\Z^{d+1}$ and its multilinear extension $g$. For any $x_1 \in \R$ with $k\gt \le x_1 \le (k+1)\gt$ and for any $y \in \gt\Z^{d}$, define 
    $$f_{x_1}(y) := \frac{x_1 - k\gt}{\gt}f(k\gt, y) + \frac{(k+1)\gt -x_1}{\gt}f((k+1)\gt, y).$$ Let $g_{x_1}$ 
    be the multilinear extension of $f_{x_1}$ and observe that by definition $g_{x_1}(y) = g(x_1, y)$. Now we evaluate the norm \begin{align*}
        \int_{\R^{d+1}} g(x)^2 \d x &= \int_\R \int_{\R^{d}} g(x_1, y)^2 \d y \d x_1 = \int_\R \int_{\R^{d}} g_{x_1}(y)^2 \d y \d x_1 \\
        &\ge \int_\R \lpr{\frac{2\gt}{3}}^{d} \norm{f_{x_1}}^2 \d x_1 && \text{ind.\ hyp.} \\
        &= \lpr{\frac{2\gt}{3}}^{d} \int_\R \sum_{y \in \gt\Z^{d}} f_{x_1}(y)^2 \d x_1 \\
        &= \lpr{\frac{2\gt}{3}}^{d} \sum_{y \in \gt\Z^{d}} \sum_{k \in \Z} \int_{k\gt}^{(k+1)\gt} \frac{1}{\gt^2} \lpr{(x_1 - k\gt)f(k\gt, y) + ((k+1)\gt - x_1)f((k+1)\gt, y)}^2 \d x_1 && \text{Tonelli} \\
        &\ge \lpr{\frac{2\gt}{3}}^{d} \sum_{y \in \gt\Z^{d}} \sum_{k \in \Z} \frac{1}{\gt^2} \int_{k\gt}^{(k+1)\gt} (x_1 - k\gt)^2f(k\gt, y)^2 + ((k+1)\gt - x_1)^2f((k+1)\gt, y)^2 \d x_1 && \text{positivity} \\
        & = \lpr{\frac{2\gt}{3}}^{d} \cdot \frac{\gt}{3} \sum_{y \in \gt\Z^d}\sum_{k\in\Z} f(k\gt, y)^2 + f((k+1)\gt, y)^2 \\
        & = \lpr{\frac{2\gt}{3}}^{d+1} \norm{f}^2.
    \end{align*}
\end{proof}

\begin{lemma}\label{lem:lap comp}
    The quadratic forms of $f$ and $g$ with the respective Laplacian operators satisfy $$\<g, -\Delta g \> \le \theta^{d-2} \< f, Lf \>.$$
\end{lemma}

\begin{proof}
From \eqref{eq:fourier tech results},
\begin{align*}
    \<f, Lf\> = 4 \sum_{k=0}^d km_k && \text{and} && \<g, -\Delta g\> = 4 \gt^{d-2} \sum_{k=0}^d \lpr{\frac{1}{3}}^{k-1}km_k.
\end{align*}
Immediately we see 
$$\<g, -\Delta g\> \le \gt^{d-2} \<f, Lf\>.$$

\end{proof}

\begin{lemma} \label{lem:pot comp}
    Let $V: \R ^d \rightarrow \R _{\ge 0}$, not necessarily periodic with $\norm{V}_\infty\le1$, and let $\gt\in(0,1)$. Suppose that there exists $a \in \mathbb{R}$ such that for all $y\in\gt\Z^d$, 
    $$\frac{1}{\gt^d}\int_{\ol{C}(y)}V(x)^2 \d x \le a \cdot \frac{1}{2^d} \sum\limits_{k \in C(y)}V(k)^2.$$ The quadratic forms of $f$ and $g$ with the respective potential operators satisfy $$\<g, Vg \> \le \gt^d \lpr{10 \cdot 2^{d-1} \<f, Lf\> + \frac{a}{r} \<f, Vf\>}$$ where $r$ is a $d$-dependent constant between $\frac{1}{4}$ and $1$.
\end{lemma}

\begin{proof}
We will consider cases based on the size of the local Laplacian Rayleigh quotients. Let $\alpha(d)$ be a function of $d$ to be determined later. Let $y\in \theta \Z^d$ and let $\<f, Lf\>_{C(y)} := \sum\limits_{(u, v) \in E(C(y))} \lpr{f(u)-f(v)}^2$.
\smallskip

\noindent {\em Case 1.} $\displaystyle\frac{\<f,Lf\>_{C(y)}}{\norm{f}^2_{C(y)}} \ge \alpha(d)$.

We use the naive bound 
\begin{align*}
    \frac{\<g, Vg\>_{\ol{C}(y)}}{\norm{g}^2_{\ol{C}(y)}} \le 1 \hspace{5mm} \text{so} \hspace{5mm} \frac{\<g, Vg\>_{\ol{C}(y)}}{\norm{g}^2_{\ol{C}(y)}} \le \frac{1}{\alpha(d)} \frac{\<f,Lf\>_{C(y)}}{\norm{f}^2_{C(y)}}.
\end{align*}
Note that $\norm{g}^2_{\ol{C}(y)} = \gt^d \sum\limits_{S\subset[d]} \lpr{\frac{1}{3}}^{\abs{S}} \widehat{F}_y(S)^2$ and $\norm{f}^2_{C(y)} = 2^d \sum\limits_{S\subset[d]} \widehat{F}_y(S)^2$, so we immediately see $$ \frac{\norm{g}^2_{\ol{C}(y)}}{\norm{f}^2_{C(y)}} \le \frac{\gt^d}{2^d}.$$ Using this and rearranging gives that
\begin{equation}\label{eq:case 1 bound}
    \<g, Vg\>_{\ol{C}(y)} \le \frac{\norm{g}^2_{\ol{C}(y)}}{\norm{f}^2_{C(y)}} \frac{1}{\alpha(d)}\<f,Lf\>_{C(y)} \le \frac{\gt^d}{2^d} \frac{1}{\alpha(d)}\<f, Lf\>_{C(y)}.
\end{equation}

\noindent {\em Case 2.} $\displaystyle\frac{\<f,Lf\>_{C(y)}}{\norm{f}^2_{C(y)}} \le \alpha(d)$.

Assume $\norm{f}_{C(y)}^2 =1$. Let $\mu := \frac{\<f, \one\>}{2^d}$. By the Poincar\'e inequality,
\begin{align*}
    \norm{f-\mu\one}^2_{C(y)} \le \<f, Lf\> \le \alpha(d).
\end{align*}

Let $\norm{f-\mu\one}^2_{C(y), \infty}$ be the infinity norm of $f-\mu\one$ on $C(y)$. Then 
$$\norm{f-\mu\one}^2_{C(y), \infty} \le \norm{f - \mu\one}_{C(y)}^2 \le \alpha(d).$$
By its definition, $\mu = \sqrt{\E\lbr{f^2} - \Var(f)} \ge \sqrt{\frac{1}{2^d} - \frac{\alpha(d)}{2^d}}$. We see that for all $x \in C(y)$, 
$$\sqrt{\frac{1}{2^d}\lpr{1-\alpha(d)}} - \sqrt{\alpha(d)} \le f(x) \le \sqrt{\frac{1}{2^d}} + \sqrt{\alpha(d)}.$$

Now we relate the minimum and maximum values of $f$ on $C(y)$. Define $r$ by $$\sqrt{r} = \frac{\sqrt{1-\alpha(d)}-\sqrt{2^d\alpha(d)}}{1+\sqrt{2^d\alpha(d)}}.$$
Then we have $\min\limits_{k \in C(y)} f(k)^2 \ge r \max\limits_{k \in C(y)} f(k)^2$.

Combining the above,
\begin{equation}\label{eq:case 2 bound}
\<f, Vf\>_{C(y)} \ge \min\limits_{k \in C(y)} f(k)^2 \sum_{k \in C(y)} V(k) \ge \frac{2^d}{\gt^d} \frac{r}{a} \max\limits_{k \in C(y)} f(k)^2 \int_{\ol{C}(y)}V(x) \d x \ge \frac{2^d}{\gt^d} \frac{r}{a} \<g, Vg\>_{\ol{C}(y)}.
\end{equation}

From \eqref{eq:case 1 bound} and \eqref{eq:case 2 bound}, for any $y \in \gt\mathbb{Z}^d$
$$\<g, Vg\>_{\ol{C}(y)} \le \frac{\gt^d}{2^d}\frac{a}{r} \<f, Vf\>_{C(y)} + \frac{\gt^d}{2^d}\frac{1}{\alpha(d)} \<f, Lf\>_{C(y)}.$$

Now we sum over all $y \in \gt \Z^d$. Observe that in summing $\<f, Vf\>_{C(y)}$ over $y$, each vertex is counted $2^d$ times; similarly, in summing $\<f, Lf\>_{C(y)}$, each edge is counted $2^{d-1}$ times. Thus, 
\begin{align*}
\<g, Vg \> = \sum_{y\in\gt\Z^d} \<g, Vg\>_{\ol{C}(y)} &\le \frac{\gt^d}{2^d} \frac{a}{r}\sum_{y\in\gt\Z^d} \<f, Vf\>_{C(y)} + \frac{\gt^d}{2^d} \frac{1}{\alpha(d)} \sum_{y\in\gt\Z^d}\<f, Lf\>_{C(y)} \\
&= \frac{\gt^d}{2^d} \frac{a}{r} 2^d \<f, Vf\> + \frac{\gt^d}{2^d} \frac{1}{\alpha(d)} 2^{d-1} \<f, Lf\> \\
& = \gt^d \lpr{\frac{a}{r} \<f, Vf\> + \frac{1}{2\alpha(d)}\<f, Lf\>}.
\end{align*}
Letting $\alpha(d) = \frac{1}{10 \cdot 2^{d}}$ we have $r \ge \frac{1}{4}$ and we achieve the desired result.
\end{proof}

\smallskip
From \lemref{lem:norm comp}, \lemref{lem:lap comp}, and \lemref{lem:pot comp} we have 

\begin{align*}
    \frac{\< g, B_\theta g \>}{\norm{g}_2^2} &\le \frac{3^d}{2^d \theta^d \norm{f}^2} \lpr{\theta^d\lpr{1 + 10 \cdot 2^{d-1}} \<f, Lf \> + \frac{a}{r} \theta^d \< f, V f \> } \\
    &= \frac{1}{\norm{f}^2} \lpr{\lpr{\frac{3^d}{2^d} + 5 \cdot 3^d} \<f, Lf \> + \frac{3^d}{2^d} \frac{a}{r} \< f, V f \> }.
\end{align*}

This proves our claim with
$$c_d = \max \lcr{\frac{3^d}{2^d} + 5 \cdot 3^d, \frac{3^d}{2^d} \frac{a}{r}}. $$
\end{proof}

\bibliographystyle{alpha}
\bibliography{schrodinger}

\newcommand{\etalchar}[1]{$^{#1}$}
\begin{thebibliography}{BDH{\etalchar{+}}17b}

\bibitem[AB96]{arendt1996spectral}
Wolfgang Arendt and Charles~J.K. Batty.
\newblock {T}he spectral bound of {S}chr{\"o}dinger operators.
\newblock {\em Potential Analysis}, 5:207--230, 1996.

\bibitem[AB97]{arendt1997spectral}
Wolfgang Arendt and Charles~J.K. Batty.
\newblock {T}he spectral function and principal eigenvalues for
  {S}chr{\"o}dinger operators.
\newblock {\em Potential Analysis}, 7(1):415--436, 1997.

\bibitem[Akk10]{akkouche2010spectral}
Sofiane Akkouche.
\newblock {T}he spectral bounds of the discrete {S}chr{\"o}dinger operator.
\newblock {\em Journal of Functional Analysis}, 259(6):1443--1465, 2010.

\bibitem[BDH{\etalchar{+}}17a]{bump2017exercise}
Daniel Bump, Persi Diaconis, Angela Hicks, Laurent Miclo, and Harold Widom.
\newblock {A}n {E}xercise (?) in {F}ourier {A}nalysis on the {H}eisenberg
  {G}roup.
\newblock In {\em Annales de la Facult{\'e} des sciences de Toulouse:
  Math{\'e}matiques}, volume~26, pages 263--288, 2017.

\bibitem[BDH{\etalchar{+}}17b]{bump2017useful}
Daniel Bump, Persi Diaconis, Angela Hicks, Laurent Miclo, and Harold Widom.
\newblock {U}seful bounds on the extreme eigenvalues and vectors of matrices
  for {H}arper's operators.
\newblock {\em Large Truncated Toeplitz Matrices, Toeplitz Operators, and
  Related Topics: The Albrecht B{\"o}ttcher Anniversary Volume}, pages
  235--265, 2017.

\bibitem[BVZ97]{beguin1997spectrum}
C{\'e}dric B{\'e}guin, Alain Valette, and Andrzej Zuk.
\newblock {O}n the spectrum of a random walk on the discrete {H}eisenberg group
  and the norm of {H}arper's operator.
\newblock {\em Journal of Geometry and Physics}, 21(4):337--356, 1997.

\bibitem[BWZ24]{witt}
Simon Becker, Jens Wittsten, and Maciej Zworski.
\newblock {\em {D}iscrete vs.\ {C}ontinuous in the {S}emiclassical {L}imit}.
\newblock 2024.

\bibitem[BZ05]{boca2005norm}
Florin~P. Boca and Alexandru Zaharescu.
\newblock {N}orm estimates of almost {M}athieu operators.
\newblock {\em Journal of Functional Analysis}, 220(1):76--96, 2005.

\bibitem[Dav99]{davies1999ground}
E.~Brian Davies.
\newblock {G}round state energy of almost periodic {S}chroedinger operators.
\newblock {\em Ergodic Theory and Dynamical Systems}, 19(3):591--609, 1999.

\bibitem[Eva22]{evans2022partial}
Lawrence~C. Evans.
\newblock {\em {P}artial {D}ifferential {E}quations}, volume~19.
\newblock American Mathematical Society, 2022.

\bibitem[GGS92]{gesztesy1992ground}
Fritz Gesztesy, Gian~Michele Graf, and Barry Simon.
\newblock {T}he ground state energy of {S}chr{\"o}dinger operators.
\newblock {\em Communications in Mathematical Physics}, 150:375--384, 1992.

\bibitem[Jit21]{jito}
Svetlana Jitomirskaya.
\newblock On point spectrum of critical almost {M}athieu operators.
\newblock {\em Advances in Mathematics}, 2021.

\bibitem[Kat52]{kato1952note}
Tosio Kato.
\newblock {N}ote on the least eigenvalue of the {H}ill equation.
\newblock {\em Quarterly of Applied Mathematics}, 10(3):292--294, 1952.

\bibitem[KKN21]{kaluba2021property}
Marek Kaluba, Dawid Kielak, and Piotr~W Nowak.
\newblock On property (t) for aut(f\_n) and sl\_n(z).
\newblock {\em Annals of Mathematics}, 193(2):539--562, 2021.

\bibitem[KNO19]{kaluba2019aut}
Marek Kaluba, Piotr~W Nowak, and Narutaka Ozawa.
\newblock Aut (f \_5) aut (f 5) has property (t).
\newblock {\em Mathematische annalen}, 375:1169--1191, 2019.

\bibitem[Mil13]{milatovic2013spectral}
Ognjen Milatovic.
\newblock {A} spectral property of discrete {S}chr{\"o}dinger operators with
  non-negative potentials.
\newblock {\em Integral Equations and Operator Theory}, 76:285--300, 2013.

\bibitem[MJ17]{marx2017dynamics}
Christoph~A. Marx and Svetlana Jitomirskaya.
\newblock {D}ynamics and spectral theory of quasi-periodic
  {S}chr{\"o}dinger-type operators.
\newblock {\em Ergodic Theory and Dynamical Systems}, 37(8):2353--2393, 2017.

\bibitem[O'D14]{o2014analysis}
Ryan O'Donnell.
\newblock {\em {A}nalysis of {B}oolean {F}unctions}.
\newblock Cambridge University Press, 2014.

\bibitem[Oza16]{ozawa13}
Narutaka Ozawa.
\newblock {N}oncommutative real algebraic geometry of {K}azhdan's property
  ({T}).
\newblock {\em Journal of the Institute of Mathematics of Jussieu}, 15:85--90,
  2016.

\bibitem[Oza22]{ozawa2022substitute}
Narutaka Ozawa.
\newblock {A} substitute for {K}azhdan's property ({T}) for universal
  non-lattices.
\newblock {\em arXiv preprint arXiv:2207.05272}, 2022.

\bibitem[RS81]{reed1981functional}
Michael Reed and Barry Simon.
\newblock {\em {I}: {F}unctional {A}nalysis}, volume~1.
\newblock Academic press, 1981.

\end{thebibliography}
\appendix

\section{Deferred proofs}\label{sec:sobolev proof}

\begin{proof}[Proof of {\propref{prop:muBlb intro}}] 
Assume $f:\R^d\rightarrow\R_{\ge 0}$ with $\norm{f}^2 = 1$ satisfies
$$ \theta^2 \norm{\nabla f}^2 + \int_{\R^d} V(x)f(x)^2\d x\le \theta^r,$$
for $r$ a constant depending on $p$ and $d$ to be chosen later.
Let
$$ W:=\{x:V(x)\le 2\theta^r\}$$
and notice by Markov's inequality (applied to the probability distribution defined by the mass of $f$) that
$$ \sum_C \norm{f|_C}^2 \frac{\norm{f|_{\oW\cap C}}^2}{\norm{f|_C}^2}=\norm{f|_{\oW}}^2 \le\half,$$
where the sum is over unit volume cubes in $\R^d$. 
Let $S$ be the set of cubes $C$ satisfying
$$ \frac{\norm{f|_{\overline{W}\cap C}}^2}{\norm{f|_C}^2}\le\frac{3}{4},$$
and notice by Markov's inequality that 
$$\sum_{C\in S}\norm{f|_C}^2\ge c$$
for a constant $c$. Let $q>2$. For each cube $C\in S$ we have $\norm{f|_{W\cap C}}^2\ge \frac{\norm{f|_C}^2}{4}$ as well as
$$\frac{\norm{f|_{W\cap C}}_q^2}{\norm{f|_{W\cap C}}^2}\ge \vol(W\cap C)^{\frac{2}{q}-1},$$
which together yield
$$ \norm{f}_q^2 \ge \sum_{C\in S} \norm{f|_{W\cap C}}_q^2 \ge \frac{1}{4}\sum_{C\in S} \vol(W\cap C)^{\frac{2}{q}-1}\norm{f|_C}^2.$$
Since $f$ has isolated critical points and is uniformly $p$-coercive, we have the uniform estimate
$$ \vol(W\cap C)=O\lpr{\theta^{\frac{rd}{p}}}$$
for each cube $C$. Substituting, we have
$$ \norm{f}_q^2 \ge \Omega(1)\cdot \theta^{\frac{rd}{p}\lpr{\frac{2}{q}-1}}\sum_{C\in S}\norm{f_C}^2 = \Omega(1)\cdot \theta^{\frac{rd}{p}\lpr{\frac{2}{q}-1}}.$$
\begin{enumerate}
\item[(b)] If $d\ge 3$ set $q=\frac{2d}{d-2}$ and apply the Sobolev inequality to obtain
$$ \norm{\nabla f}^2= \Omega(1)\cdot \theta^{-\frac{2r}{p}},$$
whence $\theta^2\norm{\nabla f}^2=\Omega\lpr{\theta^{2-\frac{2r}{p}}}.$ Choosing $r=2-\frac{2r}{p}=\frac{2p}{p+2}$ yields the claim.
\item[(a)] If $d=1,2$ we instead use the Sobolev inequalities
$$ \norm{\nabla f}^2 + \norm{f}^2\ge C_q\norm{f}_q^2,$$
for $q=\infty$ and $q$ large, respectively. Since $\norm{f}^2=1$ and the right hand side is unbounded as $\theta\rightarrow 0$, this term can be ignored and we have $\norm{f}_q^2\ge \Omega(1)\cdot \theta^{-\frac{r}{p}}$ for $d=1$ and $\norm{f}_q^2\ge \Omega(1)\cdot \theta^{-\frac{2r}{p}+o(1)}$ for $d=2$. Choosing $r=2-\frac{r}{p} $ and $r=2-\frac{2r}{p}+o(1)$ yields $r=\frac{2p}{p+1}$ and $r=\frac{2p}{p+2}+o(1)$, as advertised.
\end{enumerate}

Now, we set out to prove the upper bound (c) through a different approach. By boundedness, there is a constant $H$ such that $V(x)\le H\norm{x}^p=:W(x)$ where we assume that $V(0)=0$ and that 0 is a critical point of order $p$. For all test functions $g$, $\cR_{B_\gt}(g)\le\cR_{-\gt^2\Delta+W}(g)$. Let 
$$g_\gs(x)=\exp\lpr{-\frac{\norm{x}^2}{2\gs^2}}.$$ 
Then, 
$$\<g_\gs,Wg_\gs\>=H\int_{\R^d}\norm{x}^pg_\gs(x)\d x=HS_{d-1}\int_{\R_{\ge0}} r^{p+d-1}\exp\lpr{-\frac{r^2}{2\gs^2}}\d r=HS_{d-1}2^{\frac{p+d-2}{2}}\lpr{\frac{p+d-2}{2}}!\gs^{p+d}.$$
where $S_{d-1}=\frac{2\sqrt{\pi}^d}{\lpr{\frac{d-2}{2}}!}$ is the surface area of $\bS^{d-1}$. 
We also compute directly that
\begin{align*}
    \<g_\gs,\lpr{-\gt^2\Delta+H}g_\gs\>&=\half\lpr{d\frac{\gt^2}{\gs^2}+\tr H\gs^2}\\
    \norm{g_\gs}^2&=\lpr{\sqrt{\pi}\gs}^d
\end{align*}
so that all told, 
$$\<g_\gs,\lpr{-\gt^2\Delta+W}g_\gs\>=\half d\frac{\gt^2}{\gs^2}+2^{\frac{p+d}{2}}H\frac{\lpr{\frac{p+d-2}{2}}!}{\lpr{\frac{d-2}{2}}!}\gs^p.$$
By weighted AM--GM this is minimized by 
\begin{equation}
    \frac{2d}{p}\sqrt[\frac{p}{2}+1]{\frac{p}{d}2^{\frac{d-2}{2}}H\frac{\lpr{\frac{p+d-2}{2}}!}{\lpr{\frac{d-2}{2}}!}}\gt^{\frac{2p}{p+2}}\label{eq:muBub exact}
\end{equation}
where the exponent of $\gt$ matches the claim. 
\end{proof}

\begin{proof}[Proof of {\lemref{lem:shiftinv}}] 
Observe that $\norm{A_{\gt,0}-A_{\gt,\eta}}\le \sup\limits_{n\in\Z^d}\abs{V(n\theta+\eta)-V(n\theta)}\le \omega(\eta)$ where $\omega(\cdot)$ is the modulus of continuity of $V$, so $\mu(A_{\gt,\eta})$ is continuous in $\eta$. 

Let $\eps > 0$ and $\eta\in\R^d$. Since $\theta$ is irrational, ergodicity in each coordinate of $n\theta+\eta$ implies that there is some $\eta'\in\R^d$ and $n_0\in\Z^d$ such that $\abs{\eta'-\eta}\le\eps$ and $n_0\theta \equiv \eta'\pmod{1}$ (solve for an appropriate entry of $n_0$ separately for each component). Observe that $A_{\gt,\eta'}=A_{\gt,0}$ since the Laplacian part of $A_\gt$ is translation-invariant and we have $V(n_0\theta+n\theta)=V(n\theta+\eta')$. By continuity we now have $\abs{\mu(A_0)-\mu(A_\eta)}\le\eps$. Letting $\eps\to0$ yields the result.
\end{proof}

\end{document}